\documentclass{amsart}

\usepackage{amssymb}

\usepackage{amsfonts}
\usepackage{mathrsfs}
\usepackage{bbm}

\theoremstyle{plain}
    \newtheorem{theorem}{Theorem}[section]
    \newtheorem{lemma}[theorem]{Lemma}

\theoremstyle{definition}
    \newtheorem{definition}{Definition}[section]
    \newtheorem{remark}{Remark}[section]
    \newtheorem*{acknowledgement}{Acknowledgement}
\theoremstyle{remark}
    
\numberwithin{equation}{section}

\renewcommand{\l}{\left}
\renewcommand{\r}{\right}
\newcommand{\cleq}{\lesssim}

\newcommand{\eps}{\varepsilon}
\def\norm#1{\left\Vert #1 \right\Vert} 

\newcommand{\C}{\mathbb{C}}

\newcommand{\R}{\mathbb{R}}

\DeclareMathOperator{\im}{Im}
\DeclareMathOperator{\re}{Re}

\begin{document}

\title[Blow-up of NLS system without mass-resonance]{Blow-up of the radially symmetric solutions for the quadratic nonlinear Schr\"{o}dinger system without mass-resonance}

\author[T.Inui]{Takahisa Inui}
\address{Department of Mathematics, Graduate School of Science, Osaka University, Toyonaka, Osaka, 560-0043, Japan}
\email{inui@math.sci.osaka-u.ac.jp}

\author[N.Kishimoto]{Nobu Kishimoto}
\address{Research Institute for Mathematical Sciences, Kyoto University, Kyoto, 606-8502, Japan}
\email{nobu@kurims.kyoto-u.ac.jp}

\author[K.Nishimura]{Kuranosuke Nishimura}
\address{Department of Mathematics, Graduate School of Science, Tokyo University of Science, 1-3 Kagurazaka, Shinjuku-ku, Tokyo 162-8601, Japan}
\email{1117614@ed.tus.ac.jp}

\date{\today}
\maketitle

\begin{abstract}
We consider the quadratic nonlinear Schr\"{o}dinger system
\begin{align*}
	\begin{cases}
		i\partial_t u + \Delta u = v \overline{u}, 
		\\
		i\partial_t v+\kappa \Delta v = u^2, 
	\end{cases}
	\text{ on } I \times \R^d,
\end{align*}
where $1\leq d \leq 6$ and $\kappa>0$. In the lower dimensional case $d=1,2,3$, it is known that the $H^1$-solution is global in time. On the other hand, there are finite time blow-up solutions when $d=4,5,6$ and $\kappa=1/2$. The condition of $\kappa=1/2$ is called mass-resonance. In this paper, we prove finite time blow-up under radially symmetric assumption when $d=5,6$ and $\kappa \neq 1/2$ and we show blow-up or grow-up when $d=4$. 
\end{abstract}

\tableofcontents


\section{Introduction}
\subsection{Background}

We consider
\begin{align}
\label{NLS'}
	\begin{cases}
	i\partial_t u+\frac{1}{2m}\Delta u=\lambda \overline{u}v,
	\\
	i\partial_t v+\frac{1}{2M}\Delta v=\mu u^2,
	\end{cases}
	\text{ on } I \times \R^d,
\end{align}
where $1\leq d \leq 6$, $(u,v)$ is a $\C^2$-valued unknown function, and $m,M>0$, $\lambda ,\mu \in \C \setminus \{0\}$ are constants.
From physical viewpoint, \eqref{NLS'} is related to the Raman amplification in a plasma. See \cite{CoCoOh09} for details. 
The equation \eqref{NLS'} is invariant under the scaling $\alpha^2(u,v)(\alpha^2t,\alpha x)$ for $\alpha >0$. From this point of view, the critical regularity of the Sobolev space is $s_c=d/2-2$. Therefore, the equation \eqref{NLS'} is called $L^2$-subcritical if $d\leq3$, $L^2$-critical if $d=4$, $\dot{H}^{1/2}$-critical if $d=5$, and $\dot{H}^1$-critical if $d=6$. If $\lambda =c\bar{\mu}$ for some $c>0$, then the mass and the energy are conserved.
In this paper, we focus on the $L^2$-critical and $L^2$-supercritical case with conservation laws, \textit{i.e.}, $d=4,5,6$ and $\lambda =c\bar{\mu}$.
By considering the equation for $(\sqrt{c}|\mu |u(t,x/\sqrt{2m}),c\bar{\mu}v(t,x/\sqrt{2m}))$, we may assume $m=1/2$, $\lambda =\mu =1$.
Thus, we consider the following quadratic nonlinear Schr\"{o}dinger system: 
\begin{align}
\label{NLS}
\tag{NLS}
	\begin{cases}
		i\partial_t u + \Delta u = v \overline{u}, \qquad \text{ on } I \times \R^d,
		\\
		i\partial_t v+\kappa \Delta v = u^2,  \qquad \text{ on } I \times \R^d,
		\\
		(u(0),v(0))=(u_0,v_0) \in H^1(\R^d)\times H^1(\R^d),
	\end{cases}
\end{align}
where $\kappa>0$. The equation \eqref{NLS} has two conserved quantities, \textit{i.e.}, the mass and the energy, which are defined by 
\begin{align}
	\tag{Mass} 
	M(u,v)&:= \norm{u}_{L^2}^2 + \norm{v}_{L^2}^2,
	\\
	\tag{Energy}
	E(u,v)&:= \norm{\nabla u}_{L^2}^2 + \frac{\kappa}{2} \norm{\nabla v}_{L^2}^2 + \re \int_{\R^d} \overline{v}u^2 dx.
\end{align}
The local well-posedness in $H^1(\R^d) \times H^1(\R^d)$ for $1\leq d \leq 6$, the global well-posedness in $H^1(\R^d) \times H^1(\R^d)$ for $1\leq d \leq 3$ ($L^2$-subcritical), and the existence of the ground sate standing wave solutions for $1\leq d \leq 6$ were shown by Hayashi, Ozawa, and Tanaka \cite{HaOzTa13}. We recall the ground state. When $1\leq d \leq 5$, the system \eqref{NLS} has a standing wave solution of the form
\begin{align}
\label{eq1.2}
	(u,v)=( e^{it}\phi (x), e^{2it}\psi (x))
\end{align}
with $\R$-valued functions $\phi ,\psi$.
In fact, if \eqref{eq1.2} is a solution of \eqref{NLS}, then $(\phi ,\psi )$ should satisfy the following system of elliptic equations 
\begin{equation}
\label{SE}
\begin{cases}
	-\phi +\Delta \phi =\phi\psi ,& \text{ in  } \R^d,
	\\
	-2\psi +\kappa \Delta \psi =\phi ^2, & \text{ in  } \R^d.
\end{cases}
\end{equation}
On the other hand, when $d=6$, the system \eqref{NLS} has a static solution of the form
\begin{align}
\label{eq1.4}
	(u,v)=( \phi (x), \psi (x))
\end{align}
with $\R$-valued functions $\phi ,\psi$.
In fact, if \eqref{eq1.4} is a solution of \eqref{NLS}, then $(\phi ,\psi )$ should satisfy the following system of elliptic equations 
\begin{equation}
\label{SE*}
\begin{cases}
	\Delta \phi =\phi\psi ,& \text{ in  } \R^6,
	\\
	\kappa \Delta \psi =\phi ^2, & \text{ in  } \R^6.
\end{cases}
\end{equation}
The solutions of these elliptic equations \eqref{SE} and \eqref{SE*} can be characterized by the variational argument. Namely, the minimal mass-energy solutions exist and they are called ground states. Roughly speaking, they are characterized by the Pohozaev functional $K$, which is defined by
\begin{align*}
	K( u, v )
	=K_{d}( u, v )
	:= \norm{\nabla u}_{L^2}^2+\frac{\kappa}{2} \norm{\nabla v}_{L^2}^2 
	+\frac{d}{4}\re \int_{\R^d} \overline{v}u^2dx.
\end{align*}
We note that $K(\phi,\psi)=0$ if $(\phi ,\psi )$ is a solution of \eqref{SE} or \eqref{SE*}. Hayashi, Li, and Ozawa \cite{HaLiOz11} investigated the small data scattering. Recently, scattering below the ground state was also obtained by Hamano \cite{Ham18p} when $d=5$ and the authors \cite{InKiNi18p} when $d=4$, where scattering means that the solution of nonlinear system \eqref{NLS} approches to a free solution to the Schr\"{o}dinger equations as time goes to infinity. 

Moreover, the blow-up phenomena of the solutions to \eqref{NLS} with $\kappa=1/2$ is also investigated by many researchers. 
When $d=4$, Hayashi, Ozawa, and Tanaka proved that the solution of \eqref{NLS} starting from any initial data $(u_0,v_0)\in (H^1(\R^4)\cap L^2(\R ^4,|x|^2dx))^2=: \Sigma(\R^4)$ with $E(u_0,v_0)<0$ must blow up in finite time (\cite{HaOzTa13}). See also \cite{Din18p} for the blow-up of the radial solutions.  Recently, when $d=5$, Hamano \cite{Ham18p} proved that the solution with $(u_0,v_0) \in \Sigma(\R^5)$ or with radial symmetry blows up if the initial data satisfies $E(u_0,v_0)M(u_0,v_0)<E(\phi,\psi)M(\phi,\psi)$ and $K(u_0,v_0)<0$. He also showed the blow-up or grow-up result for non-radial solutions under $E(u_0,v_0)M(u_0,v_0)<E(\phi,\psi)M(\phi,\psi)$ and $K(u_0,v_0)<0$. 
These blow-up results were obtained under the mass-resonance condition, \textit{i.e.}, $\kappa=1/2$, since the virial identity is similar to the corresponding single nonlinear Schr\"{o}dinger equation.  
In this paper, we are interested in the blow-up phenomena when $\kappa\neq 1/2$ and $d=4,5,6$. In this case, we have to control an extra term which does not appear when $\kappa=1/2$. 


\subsection{Main results}

In this section, we give main results in this paper. 
We obtain the following blow-up result for the radial solutions when $d=5,6$.

\begin{theorem}
\label{thm1.1}
Let $d=5,6$, $\kappa \neq 1/2$, and $(\phi,\psi)$ be a ground state. Assume that $(u_0,v_0) \in H^1(\R^d) \times H^1(\R^d)$ is radially symmetric and satisfies 
\begin{itemize}
\item[(A$_5$)] if $d=5$, $E(u_0,v_0)M(u_0,v_0)<E(\phi,\psi)M(\phi,\psi)$ and $K(u_0,v_0)<0$,
\item[(A$_6$)] if $d=6$, $E(u_0,v_0)<E(\phi,\psi)$ and $K(u_0,v_0)<0$.
\end{itemize}
Then, the solution must blow up in both time directions. 
\end{theorem}

\begin{remark}
After the submission of this paper, the authors have learned that Yoshida obtained a similar blow-up result in his unpublished paper \cite{Yos13}. He considered the corresponding three-component NLS system without the mass-resonance condition and proved the finite time blow-up of radially symmetric solutions with negative energy when $d=5,6$. Part of our argument in the proof of Theorem \ref{thm1.1} is in fact very similar to his. One of our contribution is to show blow-up under the variational setting, which means we do not need to assume negative energy, and thus the strong instability of the radial ground states is also shown. See also \cite{OzSu13} for the blow-up of solutions with negative energy for the corresponding three-component NLS system with mass-resonance and interesting blow-up phenomena for other nonlinear Schr\"{o}dinger systems without mass-resonance.
\end{remark}

Before stating second main result, we give the definition of grow-up.

\begin{definition}[Grow-up]
We say that the solution $(u,v)$ grows up in positive (negative) time direction if the solution exists globally in positive (negative) time direction and 
\begin{align*}
	\limsup_{t \to \infty(-\infty)} \norm{\nabla u(t)}_{L^2} = \infty
	\quad \text{ and } \quad 
	\limsup_{t \to \infty(-\infty)} \norm{\nabla v(t)}_{L^2} = \infty
\end{align*}
\end{definition}

We obtain the blow-up or grow-up result when $d=4$.  

\begin{theorem}
\label{thm1.2}
Let $d=4$, $\kappa \neq 1/2$. Assume that $(u_0,v_0) \in H^1(\R^d) \times H^1(\R^d)$ is radially symmetric and satisfies
$E(u_0,v_0)<0$,
then, the solution $(u,v)$ with the initial data $(u_0,v_0)$ blows up or grows up in both time directions. 
\end{theorem}

\begin{remark}
In Theorems \ref{thm1.1} and \ref{thm1.2}, the assumption $\kappa\neq1/2$ is not needed, that is, we do not use $\kappa \neq 1/2$ and thus we can apply our proofs to the case of $\kappa=1/2$. See \cite{HaOzTa13,Ham18p,Din18p} for the other proofs of blow-up when $\kappa=1/2$. 
\end{remark}

\begin{remark}
These results also trivially mean the instability of the radial ground states, especially their strong instability when $d=5,6$. 
The strong instability of the ground state means the existence of a finite time blow-up solution starting from an arbitrarily small neighborhood of the ground state. Since the ground state has strictly positive energy when $d=5,6$, its strong instability does not follow from the blow-up result for solutions with negative energy. Meanwhile, we can easily find the initial data satisfying (A$_5$) or (A$_6$) in an arbitrary neighborhood of the ground state by rescaling it. We also remark that, in the $d=5$, radial, and mass-resonance case, strong instability of the ground state for the system \eqref{NLS} follows from the aforementioned result by Hamano \cite{Ham18p}.
\end{remark}


\subsection{Idea of Proof} We recall the blow-up result for the corresponding single NLS
\begin{align*}
	i \partial_t u + \Delta u +|u|u=0, \quad \text{ on } I \times \R^d,
\end{align*}
where $4\leq d \leq 6$. 
By the virial identity 
\begin{align*}
	\frac{d^2}{dt^2} \l( \int_{\R^d} |x|^2 |u(t,x)|^2 dx \r) =8 \l( \norm{\nabla u(t)}_{L^2}^2 - \frac{d}{6} \norm{u(t)}_{L^{3}}^{3}\r),
\end{align*}
whose right hand side corresponds to the Pohozaev functional, 
Glassey \cite{Gla77} showed the blow-up when $u_0 \in H^1(\R^d) \cap L^2(\R^d,|x|^2dx)$. Even if $u_0 \notin H^1(\R^d) \cap L^2(\R^d,|x|^2dx)$, Ogawa and Tsutsumi \cite{OgTs91} proved blow-up by a localized virial identity under radial symmetry. See \cite{HoRo10, AkNa13,DuWuZh16} for recent progress. For the system \eqref{NLS}, we have the following virial identity
\begin{align*}
	&\frac{d^2}{dt^2} \l(\int_{\R^d} |x|^2 |u(t,x)|^2 dx + \int_{\R^d} \frac{1}{2\kappa} |x|^2 |v(t,x)|^2 dx \r)
	\\
	&=\frac{d}{dt} \l\{ 4  \im \int x \cdot \l( \bar{u}\nabla u + \frac{1}{2} \bar{v}\nabla v \r) dx  +\left(2-\frac{1}{\kappa}\right) \im  \int_{\R^d}  |x|^2 v(t,x)\overline{u(t,x)}^2 dx  \r\}	
	\\
	&=8K(u(t),v(t)) + \l(2-\frac{1}{\kappa}\r) \frac{d}{dt} \im   \int_{\R^d}  |x|^2 v(t,x)\overline{u(t,x)}^2 dx. 
\end{align*}
Ozawa and Sunagawa obtained such a virial identity in \cite[Appendix A]{OzSu13}. 
If $\kappa=1/2$, the extra term (the second term of the last) does not appear and thus the similar contradiction argument to that for the single NLS works well. On the other hand, the extra term appears in the virial identity when $\kappa\neq 1/2$. 

When $d=5,6$, in the proof of Theorem \ref{thm1.1}, we proceed the argument without the extra term. We do not treat the extra term since we only use a localized version of 
\begin{align}
\label{eq1.6}
	\frac{d}{dt} \l\{ 4  \im \int x \cdot \l( \bar{u}\nabla u + \frac{1}{2} \bar{v}\nabla v \r) dx \r\} =8K(u(t),v(t)).
\end{align}
We use radial symmetry only to control the error term which comes from the localization. When $d=5,6$, we expect that $K$ behaves like $-(\| \nabla u  \|_{L^2}^2 + \frac{\kappa}{2}  \| \nabla v \|_{L^2}^2)$ since $8K(u,v)= 2d E(u,v) - 2(d-4) (  \| \nabla u  \|_{L^2}^2 + \frac{\kappa}{2}  \| \nabla v \|_{L^2}^2 )$ if the $\dot{H}^1$-norms are large. From this observation, we can derive a contradiction by a localized version of \eqref{eq1.6}. 
When $d=4$, we have no such expectation since $K(u,v)= E(u,v)$. Therefore, we can only show the blow-up or grow-up result in the $L^2$-critical case.

\section{Proof}

To prove blow-up results, we use the virial argument. 
We define
\begin{align*}
	V(t):= \int_{\R^d} \chi(x) |u(t,x)|^2 dx + \int_{\R^d} \frac{1}{2\kappa} \chi(x) |v(t,x)|^2 dx
\end{align*}
for a smooth function $\chi:\R^d \to \R_{\geq0}$.
By simple calculations, we get the following. 

\begin{lemma}[Localized virial identity]
\label{lem2.1}
We have the following. 
\begin{align*}
	V'(t)
	&=2\int_{\R^d} \nabla \chi (x) \im[\bar{u}\nabla u 
	+ \frac{1}{2} \overline{v}\nabla v](x) dx 
	+\l(2-\frac{1}{\kappa}\r) \int_{\R^d} \chi (x) \im[v\bar{u}^2] (x)dx,
	\\
	V''(t)
	&=
	4 \int_{\R^d} \chi_{jk} (x)\re [\overline{u}_j u_k + \frac{\kappa }{2} (\overline{v}_j v_k)] (x) dx 
	-\int_{\R^d} [\Delta \Delta \chi](x) [|u|^2 + \frac{\kappa}{2}|v|^2 ](x) dx  
	\\
	&\qquad +\int_{\R^d} [\Delta \chi](x) \operatorname{Re}[v\bar{u}^2](x) dx 
	+  \left(2-\frac{1}{\kappa}\right)\frac{d}{dt} \int_{\R^d} \chi (x) \im[v\bar{u}^2] (x) dx.
\end{align*}
\end{lemma}

\begin{proof}
Simple calculation gives us
\begin{align*}
	V' (t) 
	&= \int \chi (x)\left[  -2\partial_k \operatorname{Im}[\bar{u}u_k + \frac{1}{2} \bar{v}v_k ] + \left(2-\frac{1}{\kappa}\right) \im[v\bar{u}^2]\right] dx 
	\\
	&=2\int \nabla \chi (x) \cdot \operatorname{Im}[\bar{u}\nabla u + \frac{1}{2} \bar{v}\nabla v](x) dx +\left(2-\frac{1}{\kappa}\right) \int \chi (x) \im [v\bar{u}^2] (x) dx
	\\
	&=:J(t)+I(t)
\end{align*}
and
\begin{align*}
	J' (t) 
	&= 2 \sum_{j,k}\int \chi_k  (x)[-2 \operatorname{Re}[(\bar{u}_j u_k)_j +\frac{\kappa}{2} (\bar{v}_j v_k)_j]
	 +\frac{1}{2} \partial_{jjk} [|u|^2 + \frac{\kappa}{2} |v|^2]-\frac{1}{2} \operatorname{Re}(v\bar{u}^2)_k] dx
	\\
	&=4\sum_{j,k} \int \chi_{jk} (x)\re [\bar{u}_j u_k + \frac{\kappa }{2} \bar{v}_j v_k] (x)dx -\int [\Delta \Delta \chi](x) [|u|^2 + \frac{\kappa}{2}|v|^2 ](x) dx  
	\\
	&\qquad +\int [\Delta \chi](x) \operatorname{Re}[v\bar{u}^2](x)dx.
\qedhere
\end{align*}
\end{proof}

\begin{remark}
In this paper, we do not use $V$. We only treat $J$ and $J'$.
\end{remark}

We take a smooth function $\chi_{1}: \R_{\geq0} \to \R_{\geq 0}$ such that
\begin{align*}
	&\chi_{1}(r) 
	:=
	\begin{cases}
	r^2, & \text{ if } 0\leq r \leq 1,
	\\
	\text{positive constant}, & \text{ if } r \geq2,
	\end{cases}
	\qquad \chi_{1}'(r) 
	\begin{cases}
	= 2r, & \text{ if } 0\leq r \leq 1,
	\\
	\leq 2r, & \text{ if } 1\leq r \leq 2,
	\\
	=0, & \text{ if } r \geq2,
	\end{cases} 
	\\
	&\text{and } \chi_{1}''(r) \leq 2.
\end{align*}
For $R>0$, we take $\chi: \R^d \to \R_{\geq0}$ such that
\begin{align*}
	\chi(x)=R^2 \chi_{1}\l(\frac{|x|}{R}\r).
\end{align*}

To control error terms which comes from the localization, we use radial symmetry. For radial functions, we have the following lemma.
\begin{lemma}[Radial Sobolev inequality]
\label{lem2.2}
Let $d \geq 3$ and let $f \in \dot{H}^{1}(\R^d)$ be radially symmetric. Then, there exists a positive constant $C=C(d)$ such that
\begin{align*}
	\sup_{x \in \R^d} |x|^{\frac{d-2}{2}} |f(x)| \leq C \norm{\nabla f}_{L^2}.
\end{align*}
\end{lemma}

\begin{proof}
See Cho--Ozawa \cite[Proposition 1]{ChOz09}.
\end{proof}

We prove Theorem \ref{thm1.1}. In this proof, we use Lemmas \ref{lemA.1} and \ref{lemA.2} in Appendix. 

\begin{proof}[Proof of Theorem \ref{thm1.1}]
Let $d =5$ or $6$. 
We assume that $(u_0,v_0) \in H^1(\R^d) \times H^1(\R^d)$ is radially symmetric and satisfies (A$_5$) when $d=5$ or (A$_6$) when $d=6$. We use contradiction argument. Suppose that the solution is global in positive time direction.
We set $K(t):=K(u(t),v(t))$. By the localized virial identity, Lemma \ref{lem2.1}, we get
\begin{align*}
	J'(t) = 8K(t) +\mathcal{R}_1 +\mathcal{R}_2+\mathcal{R}_3,
\end{align*}
where $\mathcal{R}_1, \mathcal{R}_2, \mathcal{R}_3$ is defined by 
\begin{align*}
	\mathcal{R}_1&:=4\int_{\R^d}\l( \chi_{1}'\l(\frac{|x|}{R}\r) \frac{R}{|x|} -2\r)\l( |\nabla u|^2+\frac{\kappa}{2} |\nabla v|^2 \r)dx,
	\\
	&\qquad +4\int_{\R^d} \l\{ \chi_{1}'' \l(\frac{|x|}{R}\r) \frac{1}{|x|^2} - \chi_{1}'\l(\frac{|x|}{R}\r) \frac{R}{|x|^3}\r\} \l( |x\cdot\nabla u|^2+\frac{\kappa}{2} |x \cdot\nabla v|^2 \r)dx
	\\
	\mathcal{R}_2&:=-\int_{\R^d} [\Delta \Delta \chi](x) [|u|^2 + \frac{\kappa}{2}|v|^2 ](x) dx,
	\\
	\mathcal{R}_3&:=\int_{\R^d} \l\{ \chi_{1}'' \l(\frac{|x|}{R}\r)  + \chi_{1}'\l(\frac{|x|}{R}\r) \frac{(d-1)R}{|x|}-2d\r\} \re [v\overline{u}^2](x)dx.
\end{align*}
First, we show $\mathcal{R}_1 \leq 0$. 
We have $|x \cdot \nabla u|=|x||\nabla u|$ and $|x \cdot \nabla v|=|x||\nabla v|$ since $u$ and $v$ are radially symmetric. Therefore, we get
\begin{align*}
	\mathcal{R}_1
	&=4\int_{\R^d}\l(\chi_{1}'' \l(\frac{|x|}{R}\r)  -2\r)\l( |\nabla u|^2+\frac{\kappa}{2} |\nabla v|^2 \r)dx
	\leq 0,
\end{align*}
since $\chi_{1}''(r) \leq 2$. 
Next, we consider $\mathcal{R}_2$. 
Since $\chi_{1}'(r) = 2r$, $\chi_{1}''(r)=2$, and $\chi_{1}^{(3)}(r)= \chi_{1}^{(4)}(r)=0$ on $r \leq 1$, we have
\begin{align*}
	[\Delta \Delta \chi](x)
	&= \chi_{1}^{(4)} \l(\frac{|x|}{R}\r)\frac{1}{R^2} + 2\chi_{1}^{(3)}\l(\frac{|x|}{R}\r) \frac{d-1}{R|x|} 
	\\
	&\quad +\frac{(d-1)(d-3)}{|x|^2} \chi_{1}''\l(\frac{|x|}{R}\r)  -\frac{(d-1)(d-3)R}{|x|^3} \chi_{1}'\l(\frac{|x|}{R}\r)
	\\
	&=0 \text{ on } |x| \leq R.
\end{align*}
Therefore, $\mathcal{R}_2$ is estimated as follows.
\begin{align*}
	\mathcal{R}_2 
	&\leq \int_{|x| \geq R} [\Delta \Delta \chi](x) [|u|^2 + \frac{\kappa}{2}|v|^2 ](x) dx
	\\
	&\cleq \int_{|x| \geq R} R^{-2} [|u|^2 + \frac{\kappa}{2}|v|^2 ](x) dx
	\\
	&\leq R^{-2}C_{\kappa}M(u,v).
\end{align*}
At last, we consider $\mathcal{R}_3$. Since 
\begin{align*}
	\chi_{1}'' \l(\frac{|x|}{R}\r)  + \chi_{1}'\l(\frac{|x|}{R}\r) \frac{(d-1)R}{|x|} = 2d \text{ if } |x| \leq R,
\end{align*}
it follows from the radial Sobolev inequality, Lemma \ref{lem2.2}, that
\begin{align*}
	\mathcal{R}_3 
	&= \int_{|x| \geq R} \l\{ \chi_{1}'' \l(\frac{|x|}{R}\r)  + \chi_{1}'\l(\frac{|x|}{R}\r) \frac{(d-1)R}{|x|}-2d\r\} \re [v\overline{u}^2](x)dx
	\\
	&\cleq \int_{|x| \geq R} |v||u|^2dx
	\\
	&\cleq R^{-\frac{d-2}{2}} \int_{|x| \geq R} |x|^{\frac{d-2}{2}}|v||u|^2dx
	\\
	&\cleq R^{-\frac{d-2}{2}} \norm{v}_{L^2} \norm{u}_{L^2} \norm{ |x|^{\frac{d-2}{2}} u}_{L^{\infty}}
	\\
	&\cleq R^{-\frac{d-2}{2}} \norm{v}_{L^2} \norm{u}_{L^2} \norm{\nabla u}_{L^2}
	\\
	&\leq R^{-\frac{d-2}{2}} C M(u,v) \norm{\nabla u}_{L^2}.
\end{align*}
Combining these estimates, we get
\begin{align}
\label{eq2.1}
	J'(t) \leq 8K(t) + R^{-2}C_{\kappa}M(u,v) +R^{-\frac{d-2}{2}} C M(u,v) \norm{\nabla u}_{L^2}.
\end{align}
From the Young inequality, it follows that
\begin{align}
\label{eq2.2}
	R^{-\frac{d-2}{2}} C M(u,v) \norm{\nabla u}_{L^2}
	\leq 2(d-4) \eps \norm{\nabla u}_{L^2}^2 + C_{\eps} R^{-(d-2)} M(u,v)^2,
\end{align}
where $\eps>0$ is a small positive constant to be determined later. We set $L(t):=\norm{\nabla u}_{L^2}^2 + \frac{\kappa}{2} \norm{\nabla v}_{L^2}^2$. By $8K(t)= 2d E(u,v) - 2(d-4)L(t)$ and \eqref{eq2.2}, we obtain
\begin{align*}
	&8K(t) + R^{-\frac{d-2}{2}} C M(u,v) \norm{\nabla u}_{L^2}
	\\
	&\leq 2d E(u,v) - 2(d-4)(1-\eps)L(t)+ C_{\eps} R^{-(d-2)} M(u,v)^2
	\\
	&= 8(1-\eps) K(t) + 2d\eps E(u,v)+ C_{\eps} R^{-(d-2)} M(u,v)^2.
\end{align*}
By the variational argument, Lemmas \ref{lemA.1} and \ref{lemA.2}, we get 
\begin{align*}
	8K(t) + R^{-\frac{d-2}{2}} &C M(u,v) \norm{\nabla u}_{L^2}
	\\
	&\leq - 8(1-\eps) \delta + 2d\eps E(u,v)+ C_{\eps} R^{-(d-2)} M(u,v)^2.
\end{align*}
Taking sufficiently small $\eps>0$, which depends on $\delta$ and $E(u,v)$, it follows that
\begin{align*}
	8K(t) + R^{-\frac{d-2}{2}} C M(u,v) \norm{\nabla u}_{L^2}
	\leq - 4 \delta + C_{\eps} R^{-(d-2)} M(u,v)^2.
\end{align*}
Therefore, by this and \eqref{eq2.1}, we have
\begin{align*}
	J'(t) 
	&\leq - 4 \delta + R^{-2}C_{\kappa}M(u,v)+ C_{\eps} R^{-(d-2)} M(u,v)^2.
\end{align*}
Taking sufficiently large $R>0$ (we fix such $R$), we get
\begin{align*}
	J'(t) \leq - 2 \delta,
\end{align*}
Integrating this on $[0,t)$, we obtain
\begin{align*}
	J(t) \leq -2\delta t  +J(0).
\end{align*}
By the direct calculation, we find that
\begin{align}
\label{eq2.3}
	|J(t)| 
	&\leq 2\int |\nabla \chi (x)| [|u||\nabla u| + \frac{1}{2} |v| |\nabla v|] dx
	\\ \notag
	&\cleq R \int \chi_{1}'\l(\frac{|x|}{R}\r) [|u||\nabla u| + \frac{1}{2} |v| |\nabla v|] dx
	\\ \notag
	&\leq C_{\kappa} R M(u,v)^{\frac{1}{2}} L(t)^{\frac{1}{2}}.
\end{align}
By these inequalities, for large $t\geq T_0$, where we take $T_0$ such that $-2\delta t +J(0) < -\delta t$ for $t \geq T_0$, it follows that
\begin{align}
\label{eq2.4}
	  \delta t \leq -J(t) = |J(t)| \leq C R M(u,v)^{\frac{1}{2}} L(t)^{\frac{1}{2}}.
\end{align}
Thus, we get
\begin{align*}
	L(t) \geq Ct^2
\end{align*}
for $t>T_0$. Return to \eqref{eq2.1}. We have 
\begin{align*}
	J'(t) \leq 2d E(u,v) - 2(d-4)L(t) + R^{-2}C_{\kappa}M(u,v) +R^{-\frac{d-2}{2}} C M(u,v) \norm{\nabla u}_{L^2}.
\end{align*}
From the Young inequality, it follows that
\begin{align*}
	J'(t) \leq  - (d-4)L(t)+ 2d E(u,v)  + R^{-2}C_{\kappa}M(u,v) +R^{-(d-2)} C M(u,v)^2.
\end{align*}
Since $L(t) \geq Ct^2$ for $t>T_0$ and $E$ does not depend on $t$, there exists sufficiently large $T_1>T_0$ such that we have
\begin{align*}
	J'(t) \leq  - \frac{d-4}{2}L(t)\quad  \text{ for $t>T_1$.  }
\end{align*}
Integrating this on $[T_1,t)$, we obtain
\begin{align*}
	J(t) \leq  - \frac{d-4}{2} \int_{T_1}^{t} L(s) ds +J(T_1).
\end{align*}
Here, since $J(T_1) \leq -\delta T_1<0$, we get 
\begin{align}
\label{eq2.5}
	  \frac{d-4}{2} \int_{T_1}^{t} L(s) ds \leq -J(t) = |J(t)| \leq C R M(u,v)^{\frac{1}{2}} L(t)^{\frac{1}{2}}.
\end{align}
We set $\xi(t):= \int_{T_1}^{t} L(s) ds$. Then, \eqref{eq2.5} means that
\begin{align*}
	A \xi(t)^2 \leq \xi'(t)
\end{align*}
where $A:=(d-4)^2/(4C^2R^2M(u,v))$.
This implies that 
\begin{align*}
	A \leq \frac{\xi'(t)}{\xi(t)^2}
\end{align*}
for $t>T_1$ since $\xi(t) >0$ for all time. Integrating this on $[T_1,t)$, we obtain
\begin{align*}
	A(t-T_1) \leq \frac{1}{\xi(T_1)} - \frac{1}{\xi(t)} \leq  \frac{1}{\xi(T_1)}.
\end{align*}
Taking the limit $t \to \infty$ derives a contradiction. This means that the solution must blow up in positive time direction. Blow-up in the negative direction can be obtained similarly. This completes the proof.
\end{proof}

Next, we prove Theorem \ref{thm1.2}.

\begin{proof}[Proof of Theorem \ref{thm1.2}]
We focus on the positive direction. 
Suppose that the solution $(u,v)$ is global in positive time direction and there exists $A\in (0,\infty)$ such that
\begin{align*}
	\sup_{t>0} \norm{\nabla u}_{L^2} <A
	\quad \text{ or } \quad
	\sup_{t>0} \norm{\nabla v}_{L^2} <A.
\end{align*}
It is easy to check that if one of them is bounded then the other is also bounded by energy conservation law and the Gagliardo--Nirenberg inequality. 
From \eqref{eq2.1} in the proof of Theorem \ref{thm1.1}, we get
\begin{align*}
	J'(t) \leq 8E(u,v) + R^{-2}C_{\kappa}M(u,v) +R^{-1} C M(u,v) A.
\end{align*}
Thus, taking sufficiently large $R$, which depends on $A$, $M(u,v)$, and $\kappa$, we obtain 
\begin{align*}
	J'(t) \leq 4E(u,v)
\end{align*}
for all time. Integrating this on $[0,t)$, we have
\begin{align*}
	J(t) \leq 4E(u,v) t + J(0).
\end{align*}
And thus, we have
\begin{align*}
	J(t) \leq 2E(u,v) t
\end{align*}
for large time. 
From this inequality and \eqref{eq2.3}, it follows that
\begin{align*}
	-2E(u,v) t  \leq -J(t) =|J(t)| \leq C_{\kappa} RM(u,v)^{\frac{1}{2}}A.
\end{align*}
Taking $t \to \infty$, we get a contradiction. 
\end{proof}

\appendix
\section{Variational argument}

We collect some lemmas from variational argument. See \cite{HaOzTa13,Ham18p} for the proofs. The proofs are similar to those for the  single NLS or for the system in the case of $\kappa=1/2$.

\noindent{\bf $\bullet$ $L^2$-critical case($d=4$).}
By \cite{HaOzTa13}, it is known that there exists at least one ground state of \eqref{SE} for $\kappa>0$. They also obtained the sharp Gagliardo--Nirenberg inequality:

It holds that for any $(u,v)\in H^1(\R ^4)^2$ 
\begin{align*}
	\l| \re \int _{\R^4}u^2(x)\overline{v(x)} dx\r| 
	\leq \l( \frac{M(u,v)}{M(\phi ,\psi )}\r) ^{1/2} \l( \norm{\nabla u}_{L^2}^2+\frac{\kappa}{2}\norm{\nabla v}_{L^2}^2\r) ,
\end{align*}
where $(\phi ,\psi )$ is a ground state of \eqref{SE}. Moreover, equality is attained by the ground state.

\noindent{\bf $\bullet$ $\dot{H}^{1/2}$-critical case($d=5$).}
The ground state $( \phi, \psi )=( \phi_1, \psi_1 )$ is characterized as follows. We define
\begin{align*}
	\mu_{\omega} &:= \inf \{S_{\omega}( f, g ) :( f, g ) \in H^1(\R^5)\times H^1(\R^5) \setminus \{(0,0)\}, K( f, g )=0\}
\end{align*}
where $K=K_5$ and
\begin{align*}
	S_{\omega}( f, g )&:=\frac{1}{2}E( f, g ) +\frac{\omega}{2} M(f, g),
\end{align*}
for $\omega>0$. Then, $\mu_{\omega}= S_{\omega}(\phi_{\omega},\psi_{\omega})$ where $(\phi_{\omega}(x),\psi_{\omega}(x))=\omega(\phi(\sqrt{\omega}x), \psi(\sqrt{\omega}x))$. It is known that $E(u_0,v_0)M(u_0,v_0)<E(\phi,\psi)M(\phi,\psi)$ holds if and only if $S_{\omega}(u_0,v_0) < S_{\omega}(\phi_{\omega},\psi_{\omega})$ for some $\omega$.
By using this and variational argument, we get the following. See \cite{Ham18p}, where $\kappa =1/2$ is treated, for the detail. 

\begin{lemma}
\label{lemA.1}
Let $d=5$. If $E(u_0,v_0)M(u_0,v_0)<E(\phi,\psi)M(\phi,\psi)$ and $K(u_0,v_0)<0$, then there exists a positive constant $\delta$ such that $K(u(t),v(t)) < -\delta$ as long as the solution exists.
\end{lemma}

\noindent{\bf $\bullet$ $\dot{H}^{1}$-critical case($d=6$).}
The ground state is characterized by 
\begin{align*}
	E(\phi,\psi)= \inf\{ E(f,g): (f,g) \in \dot{H}^1(\R^6) \times\dot{H}^1(\R^6) \setminus\{(0,0)\}, K(f,g)=0 \}
\end{align*}
where $K=K_{6}$. By using this and variational argument, we get the following lemma, which is similar to Lemma~\ref{lemA.1} in the $\dot{H}^{1/2}$-critical case. 
\begin{lemma}
\label{lemA.2}
Let $d=6$. If $E(u_0,v_0)<E(\phi,\psi)$ and $K(u_0,v_0)<0$, then there exists a positive constant $\delta$ such that $K(u(t),v(t)) < -\delta$ as long as the solution exists.
\end{lemma}

\begin{acknowledgement}
The authors  would like to thank Professor Hideaki Sunagawa for his useful and important comments, especially informing them of the paper by Yoshida \cite{Yos13}. 
The first author was partially supported by JSPS Grant-in-Aid for Early-Career Scientists JP18K13444. The second author was supported in part by JSPS Grant-in-Aid for Young Scientists (B) JP16K17626. 
\end{acknowledgement}


\end{document}